\documentclass[12pt]{article}

\usepackage[a4paper, top=3cm, left=3cm, right=2cm, bottom=2cm]{geometry}
\usepackage{amsmath, amsthm}
\usepackage{amssymb}
\usepackage{multicol}
\usepackage{amsfonts}
\usepackage{graphics, color, enumerate, fancyhdr, dsfont}

    \newcommand{\Q}{\mathbb{Q}}

    \newcommand{\R}{\mathbb{R}}

    \newcommand{\ran}{\mbox{\rm ran}}

    \newcommand{\Cwf}{\mathcal{C}}
    \newcommand{\Dwf}{\mathcal{D}}

    \newcommand{\menos}{\smallsetminus}

    \newcommand{\frestr}{\!\!\upharpoonright\!\!}

\title{Coding Polish spaces}
\author{Diego Alejandro Mej\'ia}
\date{\small Faculty of Science\\ Shizuoka University\\ 836 Ohya, Suruga-ku, 422-8529 Shizuoka, Japan\\ \texttt{diego.mejia@shizuoka.ac.jp}}

\begin{document}

\makeatletter
\def\@roman#1{\romannumeral #1}
\makeatother

\theoremstyle{plain}
  \newtheorem{theorem}{Theorem}[section]
  \newtheorem{corollary}[theorem]{Corollary}
  \newtheorem{lemma}[theorem]{Lemma}
  \newtheorem{prop}[theorem]{Proposition}
  \newtheorem{claim}[theorem]{Claim}
  \newtheorem{exer}[theorem]{Exercise}
\theoremstyle{definition}
  \newtheorem{definition}[theorem]{Definition}
  \newtheorem{example}[theorem]{Example}
  \newtheorem{remark}[theorem]{Remark}
  \newtheorem{context}[theorem]{Context}
  \newtheorem{question}[theorem]{Question}
  \newtheorem{problem}[theorem]{Problem}
  \newtheorem{notation}[theorem]{Notation}

\maketitle

\newcommand{\la}{\langle}
\newcommand{\ra}{\rangle}
\newcommand{\id}{\mathrm{id}}
\newcommand{\sig}{\boldsymbol{\Sigma}}
\newcommand{\cosig}{\boldsymbol{\Pi}}

\newcommand{\leqdi}{\preceq_{\mathrm{di}}}
\newcommand{\eqdi}{\approx_{\mathrm{di}}}
\newcommand{\leqcdi}{\preceq_{\mathrm{cdi}}}
\newcommand{\eqcdi}{\approx_{\mathrm{cdi}}}
\newcommand{\eqPc}{\approx_{\mathrm{P}}}

\begin{abstract}
   We use countable metric spaces to code Polish metric spaces and evaluate the complexity of some statements about these codes and of some relations that can be determined by the codes. Also, we propose a coding for continuous functions between Polish metric spaces.
\end{abstract}

\section{Introduction}\label{SecIntro}

A \emph{Polish metric space} is a separable complete metric space $\la X,d\ra$ and a \emph{Polish space} is a topological space $X$ which is homeomorphic to some Polish metric space (in the first notion the complete metric is required). As any Polish metric space is the completion of a countable metric space and the latter can be coded by reals in $\R^{\omega\times\omega}$, we can use such reals to code Polish metric spaces. This coding was used by Clemens \cite{Clemens} to formalize the \emph{isometry relation} and to study other equivalence relations that can be reduced to that one.

In this paper, we take a closer look to this coding and study the complexity of some statements about codes, some of them characterizing relations between Polish metric spaces. In particular, we provide a different proof of \cite[Lemma 4]{Clemens} that states that the isometry relation is analytic (Theorem \ref{codesamePolishAnalytic}(f)). We also code continuous functions between Polish metric spaces by \emph{Cauchy-continuous} functions between the corresponding separable metric spaces and, like in the case of Polish metric spaces, we study the complexity of some statements about this coding. This allows us to prove that the \emph{homeomorphic relation} between codes is $\sig^1_2$ (Corollary \ref{codesamePolishSigma1-2}). The contents of this work is the starting point of research for describing certain aspects of descriptive set theory (like category and measure) by the coding presented in this paper.

We fix some notation. Given two metric spaces $\la X,d\ra$ and $\la X',d'\ra$, say that a function $\iota:\la X,d\ra\to\la X',d'\ra$ is an \emph{isometry} if, for all $x,y\in X$, $d(x,y)=d(f(x),f(y))$ (we do not demand an isometry to be onto). Additionaly, we say that $\iota$ is an \emph{isometrical isomorphism} if it is onto, for which case we say that the metric spaces $\la X,d\ra$ and $\la X',d'\ra$ are \emph{isometrically isomorphic}.

We structure this paper as follows. In Section \ref{SecMetric} we review some general aspects about completions of metric spaces. Afterwards, in Section \ref{SecPolish}, we introduce the coding for Polish metric spaces and look at the complexity of statements concerning these codes. Section \ref{SecFunc} is dedicated to the theory of codes for continuous functions between Polish metric spaces, plus some general facts about functions between metric spaces and their completions.

\section{Completion of metric spaces}\label{SecMetric}

\begin{definition}\label{Defcompletion}
   Let $\la X,d\ra$ be a metric space. Say that $\la X^*,d^*,\iota\ra$ is a \emph{completion of $\la X,d\ra$} if $\la X^*,d^*\ra$ is a complete metric space and $\iota:\la X,d\ra\to\la X^*,d^*\ra$ is a \emph{dense isometry}, that is, an isometry such that $\iota[X]$ is dense in $X^*$.
\end{definition}

Note that $d^*$ is determined by $\iota$ and $d$ because $d^*(z,z')=\lim_{n\to+\infty}d(x_n,x'_n)$ for arbitrary Cauchy sequences $\la x_n\ra_{n<\omega}$ and $\la x'_n\ra_{n<\omega}$ in $X$ such that their images on $X^*$ converge to $z$ and $z'$, respectively. It is well known that every metric space has a completion, for example, the space of its Cauchy sequences. 

Given a metric space $\la X,d\ra$ and an isometry $\iota:\la X,d\ra\to\la X^*,d^*\ra$ into a complete metric space $\la X^*,d^*\ra$, say that \emph{$\la X^*,d^*,\iota\ra$ commutes diagrams of isometries from $\la X,d\ra$} if, for any isometry $f:\la X,d\ra\to\la Y,d'\ra$ into a complete metric space $\la Y,d'\ra$, there is a unique continuous function $\hat{f}:\la X^*,d^*\ra\to\la Y,d'\ra$ such that $f=\hat{f}\circ\iota$. As a characterization of completeness of a metric space, it is well known that $\la X^*,d^*,\iota\ra$ is a completion of $\la X,d\ra$ iff it commutes diagrams of isometries, even more, such a completion is unique modulo isometries (see Lemma \ref{ComplExt}). Moreover, a completion commutes diagrams of much less than isometries.

\begin{definition}\label{DefCauchycont}
   A function $f:\la X,d\ra\to\la Y,d'\ra$ between metric spaces is \emph{Cauchy-continuous} if, for any Cauchy sequence $\la x_n\ra_{n<\omega}$ in $\la X,d\ra$, $\la f(x_n)\ra_{n<\omega}$ is a Cauchy sequence in $\la Y,d'\ra$.
\end{definition}

Clearly, any Cauchy-continuous function is continuous and any uniformly continuous function is Cauchy continuous. Also, if $f:\la X,d\ra\to\la Y,d'\ra$ is a function between metric spaces with $\la X,d\ra$ complete, then $f$ is continuous iff it is Cauchy-continuous.

\begin{theorem}\label{ComplExt}
   Let $\la X_0,d_0,\iota\ra$ be a completion of the metric space $\la X,d\ra$ and let $f:\la X,d\ra\to\la Y,d'\ra$ be a continuous function into a complete metric space $\la Y,d'\ra$.
   \begin{enumerate}[(a)]
      \item There is at most one continuous function $\hat{f}:X_0\to Y$ such that $f=\hat{f}\circ\iota$.
      \item $\hat{f}$ as in (a) exists iff $f$ is Cauchy-continuous.
      \item If $f$ is Cauchy-continuous, then
          \begin{enumerate}[({c}-1)]
             \item $\hat{f}$ is uniformly continuous iff $f$ is.
             \item $\hat{f}$ is an isometry iff $f$ is.
             \item $\hat{f}$ is an isometrical isomorphism iff $f$ is a dense isometry.
          \end{enumerate}
      \item If $\la X_1,d_1,\iota_1\ra$ commutes diagrams of isometries from $\la X,d\ra$, then there is a unique isometrical isomorphism $\iota^*:\la X_0,d_0\ra\to\la X_1,d_1\ra$ such that $\iota_1=\iota^*\circ\iota$. In particular, $\la X_1,d_1,\iota_1\ra$ is a completion of $\la X,d\ra$.
   \end{enumerate}
\end{theorem}
\begin{proof}
  \begin{enumerate}[(a)]
     \item Because $\iota[X]$ is dense in $X_0$.
     \item If $\hat{f}$ exists then it is Cauchy-continuous. As $\iota$ is Cauchy-continuous, then so is $f$.

         For the converse, we first show how to define $\hat{f}$. Given $x\in X_0$, find a sequence $\bar{x}=\la x_n\ra_{n<\omega}$ in $X$ such that $\lim_{n\to+\infty}\iota(x_n)=x$. Clearly, $\bar{x}$ is a Cauchy sequence and, as $f$ is Cauchy-continuous, $\la f(x_n) \ra_{n<\omega}$ is a Cauchy sequence in $Y$ so, by completeness, it converges in $Y$ to a point we define as $\hat{f}(x)$. Note that this point does not depend on the choice of $\bar{x}$ because, if $\bar{y}$ is another Cauchy sequence in $X$ such that $\lim_{n\to+\infty}d(x_n,y_n)=0$, then $\la x_0,y_0,x_1,y_1,\ldots\ra$ is a Cauchy sequence in $X$ and $\la f(x_0),f(y_0),f(x_1),f(y_1)\ldots\ra$ is a Cauchy sequence in $Y$, so both sequences $\la f(x_n)\ra_{n<\omega}$ and $\la f(y_n)\ra_{n<\omega}$ converge to the same point. Clearly, $f=\hat{f}\circ\iota$.

         To see the continuity of $\hat{f}$, assume that $\la x'_n\ra_{n<\omega}$ is a sequence in $X_0$ that converges to $x\in X_0$. By the definition of $\hat{f}$, for each $n<\omega$ we can find an $x_n\in X$ such that $d'(f(x_n),\hat{f}(x'_n))<2^{-(n+1)}$ and $d_0(\iota(x_n),x'_n)<2^{-(n+1)}$. Clearly, $\la\iota(x_n)\ra_{n<\omega}$ converges to $x$, so $\la f(x_n)\ra_{n<\omega}$ converges to $\hat{f}(x)$ by definition of $\hat{f}$. Therefore, $\la\hat{f}(x'_n)\ra_{n<\omega}$ converges to $\hat{f}(x)$.
     \item As $\iota$ is uniformly continuous, it is clear that $f$ is uniformly continuous if $\hat{f}$ is. For the converse, assume that $f$ is uniformly continuous and let $\varepsilon>0$. Then, there is a $\delta>0$ such that, for all $x_0,x_1\in X$, $d(x_0,x_1)<\delta$ implies $d'(f(x_0),f(x_1))<\frac{\varepsilon}{3}$. Assume that $z_0,z_1\in X_0$ and $d_0(z_0,z_1)<\frac{\delta}{3}$. For each $e=0,1$ find an $x_e\in X$ so that $d_0(\iota(x_e),z_e)<\frac{\delta}{3}$ and $d'(f(x_e),\hat{f}(z_e))<\frac{\varepsilon}{3}$. Thus $d_0(\iota(x_0),\iota(x_1))<\delta$, that is, $d(x_0,x_1)<\delta$. Then $d'(f(x_0),f(x_1))<\frac{\varepsilon}{3}$, which implies $d'(\hat{f}(z_0),\hat{f}(z_1))<\varepsilon$.

         To see (c-2), as $\iota$ is an isometry, it is clear that $f$ is an isometry if $\hat{f}$ is. For the converse, assume that $f$ is an isometry and let $x_0,x_1\in X_0$. For each $e=0,1$ find a sequence $\la x^e_n\ra_{n<\omega}$ in $X$ so that $\lim_{n\to+\infty}\iota(x^e_n)=x_e$. By continuity of metrics, it is clear that
         \begin{multline*}d_0(x_0,x_1)=\lim_{n\to+\infty}d_0(\iota(x^0_n),\iota(x^1_n))=\lim_{n\to+\infty}d(x^0_n,x^1_n)
              =\lim_{n\to+\infty}d'(f(x^0_n),f(x^1_n))\\
              =d'(\hat{f}(x_0),\hat{f}(x_1))
         \end{multline*}
         the last equality because
         \[\lim_{n\to+\infty}f(x^e_n)=\lim_{n\to+\infty}\hat{f}(\iota(x^e_n))=\hat{f}(x_e).\]
         Finally, to prove (c-3), if $f$ is a dense isometry, then so is $\hat{f}$ because $f[X]=\hat{f}[\iota[X_0]]$ is dense in $Y$. But also $\la\hat{f}[X_0],d'\ra$ is a complete metric space because $\hat{f}$ is an isometry, therefore, this set is closed in $Y$. Thus, by density, it is equal to $Y$. The converse is straightforward.
     \item As $\iota_1:\la X,d\ra\to\la X_1,d_1\ra$ is an isometry, by (b) and (c-2) there is an isometry $\iota^*:X_0\to X_1$ such that $\iota_1=\iota^*\circ\iota$. On the other hand, there is a continuous function $\iota^{**}:X_1\to X_0$ such that $\iota=\iota^{**}\circ\iota_1$. Thus $\iota=(\iota^**\circ\iota^*)\circ\iota$ and $\iota_1=(\iota^*\circ\iota^{**})\circ\iota_1$. By uniqueness of the completion of the respective diagrams, $\iota^*$ is an homeomorphism and $(\iota^*)^{-1}=\iota^{**}$. Therefore, by (c-3), $\iota_1$ is a dense isometry.
  \end{enumerate}
\end{proof}


\section{Coding Polish metric spaces}\label{SecPolish}

We code all Polish metric spaces with countable metric spaces of the form $\la \eta,d\ra$ where $\eta\leq\omega$ is an ordinal.


\begin{definition}\label{DefPolishcode}
   Let $\la\eta,d\ra$ be a metric space where $\eta$ is an ordinal $\leq\omega$.
   \begin{enumerate}[(1)]
      \item When $\la X,d_X\ra$ is a Polish metric space, we say that \emph{$\la\eta,d\ra$ codes $\la X,d_X\ra$} if $\la X,d_X,\iota\ra$ is a completion of $\la\eta,d\ra$ for some $\iota$.
      \item When $X$ is a Polish space, we say that \emph{$\la\eta,d\ra$ codes $X$} if some (or any) completion of $\la\eta,d\ra$ is homeomorphic with $X$.
   \end{enumerate}
\end{definition}

\begin{example}\label{Exmcode}
  \begin{enumerate}[(1)]
     \item The Polish metric space $\la\R,d_\R\ra$ with the standard metric is coded by $\la\omega,d_\Q\ra$ (in the sense of (1)) where the metric $d_\Q$ makes the canonical bijection $\iota_\Q:\omega\to\Q$ an isometry onto $\la\Q,d_\R\frestr(\Q\times\Q)\ra$. As a consequence, $\la\omega,d_\Q\ra$ codes $\R$ as a Polish space (in the sense of Definition \ref{DefPolishcode}(2)).
     \item For $S:\omega\to(\omega+1)\menos\{0\}$ recall the complete metric $d_{\prod S}$ on $\prod S=\prod_{n<\omega}S(n)$ given by $d_{\prod S}(x,y)=2^{-\inf\{n<\omega:x(n)\neq y(n)\}}$, which is compatible with the product topology when each $S(n)$ is discrete. Here, $\la\prod S,d_{\prod S}\ra$ is coded by $\la\eta,d_{\Q^S}\ra$ where $\eta=|\Q^S|$ with $\Q^S$ the set of eventually zero sequences in $\prod S$ and $d_{\Q^S}$ the metric on $\eta$ so that the canonical bijection $\iota_{\Q^S}:\eta\to\Q^S$ is an isometry onto $\la\Q^S,d_{\prod S}\frestr(\Q^S\times\Q^S)\ra$.
     \item As a particular case of (2), consider $\bar{\omega}:\omega\to\{\omega\}$ the constant function on $\omega$, $d_{\Q^{\bar{\omega}}}\in\Dwf(\omega)$ and the dense isometry $\iota_{\Q^{\bar{\omega}}}:\la\omega,d_{\Q^{\bar{\omega}}}\ra\to\la\omega^\omega,d_{\prod\bar{\omega}}\ra$. This is an standard coding of the Baire space.
  \end{enumerate}
\end{example}

Though Polish metric spaces coded by the same $\la\eta,d\ra$ are isometrically isomorphic, homeomorphic codes do not lead to homeomorphic Polish spaces. For example, consider the metrics $d_1$ and $d_2$ on $\omega$ where $d_1$ is the discrete metric, that is, $d_1(n,m)=1$ if $n\neq m$ or 0 otherwise, and $d_2(n,m)=|2^{-n}-2^{-m}|$. Though both metrics are compatible to the discrete topology on $\omega$, the completion of $\la\omega,d_1\ra$ is itslef, while the completion of $\la\omega,d_2\ra$ is the ordinal $\omega+1$ (with the order topology).

Note that, if $X$ is a Hausdorff topological space which contains a dense finite set, then $X$ is finite with the discrete topology, so any finite Polish space is coded by a natural number (its size) with any metric. So we only need to concentrate on Polish spaces coded by a metric on $\omega$, that is, on infinite Polish spaces.

One interesting fact is to recognize when two countable metric spaces code the same Polish metric space.

\begin{lemma}\label{sameCompl}
   Let $\la X_0,d_0\ra$ and $\la X_1,d_1\ra$ be metric spaces. Then, both metric spaces have isometrically isomorphic completions iff there exists a metric space $\la \eta,d\ra$ where $\eta$ is a cardinal $\leq|X_0|+|X_1|$ and there are dense isometries $\iota_e:X_e\to \eta$ for each $e=0,1$.
\end{lemma}
\begin{proof}
   Assume that, for each $e=0,1$, $\la X^*,d^*,\iota^*_e\ra$ is a completion of $\la X_e,d_e\ra$. Put $Y:=\iota_0[X_0]\cup\iota_1[X_1]$, $d_Y:=d^*\frestr(Y\times Y)$ and $\eta:=|Y|$. Let $g:Y\to\eta$ be a bijection and $d$ the metric on $\eta$ that makes $g$ an isometry. Thus, $\iota_e:=g\circ\iota^*_e$ is as desired.

   For the converse, assume we have such metric space $\la Y,d\ra$ and dense isometries $\iota_e$ for each $e=0,1$. It is clear that any completion of $\la Y,d\ra$ is a completion of both $\la X_0,d_0\ra$ and $\la X_1,d_1\ra$.
\end{proof}

\begin{corollary}\label{samePolishcode}
   Let $d_0$ and $d_1$ be metrics on $\omega$. The following statements are equivalent.
   \begin{enumerate}[(1)]
      \item $\la\omega,d_0\ra$ and $\la\omega,d_1\ra$ code isometrically isomorphic Polish metric spaces.
      \item There is a metric $d^*$ on $\omega$ and there is a dense isometry $\iota_e:\la\omega,d_e\ra\to \la\omega,d^*\ra$ for each $e=0,1$.
   \end{enumerate}
\end{corollary}

Let $\Dwf(\omega)$ be the set of metrics on $\omega$. Note that $\Dwf(\omega)\subseteq\R^{\omega\times\omega}$, so we can say that infinite Polish spaces are coded by \emph{reals} corresponding to metrics on $\omega$. The previous lemma indicates that codes of the same Polish metric space enjoy an amalgamation property. Define the order $\leqdi$ on $\Dwf(\omega)$ as $d\leqdi d'$ iff there is a dense isometry $\iota:\la\omega,d\ra\to \la\omega,d'\ra$ (`di' stands for `dense isometry'). So what the previous result states is that two metric spaces $\la\omega,d\ra$ and $\la\omega,d'\ra$ code the same Polish metric spaces iff there is a $d^*\in\Dwf(\omega)$ such that $d,d'\leqdi d^*$. We denote this relation by $d\eqdi d'$.

In the following result we provide the complexity of some relevant statements concerning codes for Polish metric spaces.

\newcommand{\Img}{\mathrm{Img}}

\begin{theorem}\label{codesamePolishAnalytic}
  \begin{enumerate}[(a)]
     \item The family $\Dwf(\omega)$ of metrics on $\omega$ is $\cosig^0_1$ in $\R^{\omega\times\omega}$. In particular, $\Dwf(\omega)$ is a Polish space.
     \item The statement ``$x$ is dense in the metric space $\la\omega,d\ra$" is $\sig^0_2$ in $2^\omega\times\R^{\omega\times\omega}$.
     \item The statement ``$g:\la\omega,d\ra\to\la\omega,d'\ra$ is an isometry between metric spaces" is $\cosig^0_1$ in $\omega^\omega\times(\R^{\omega\times\omega})^2$.
     \item The function $\Img:2^\omega\times\omega^\omega\to2^\omega$ defined as $\Img(x,g)=g[x]$ is continuous.
     \item The relation $\leqdi$ is $\sig^1_1$ in $(\R^{\omega\times\omega})^2$.
     \item The relation $\eqdi$ is $\sig^1_1$ in $(\R^{\omega\times\omega})^2$.
  \end{enumerate}
\end{theorem}
\begin{proof}
   $d\leqdi d'$ is equivalent to ``$d,d'\in\Dwf(\omega)$ and there exists an isometry $g:\la\omega,d\ra\to\la\omega,d'\ra$ so that $\Img(\omega,g)$ is dense in $\la\omega,d'\ra$", which is analytic by (a)-(d).
\end{proof}


Codes for perfect Polish spaces can also be classified.

\begin{lemma}\label{ComplIsol}
   Let $\la X,d\ra$ be a metric space and let $\la X^*,d^*,\iota\ra$ be its completion.
   \begin{enumerate}[(a)]
      \item If $z\in X^*$ is isolated, then $z\in\iota[X]$.
      \item $x\in X$ is isolated iff $\iota(x)$ is isolated in $X^*$.
      \item $X^*$ is perfect iff $X$ is perfect.
   \end{enumerate}
\end{lemma}
\begin{proof}
  \begin{enumerate}[(a)]
     \item Consequence of the density of $\iota[X]$.
     \item $x\in X$ is isolated iff there is some $\delta>0$ so that $\{x\}=B_X(x,\delta)$. On the other hand, for a fixed $\delta>0$, $\{x\}=B_X(x,\delta)$ iff $\{\iota(x)\}=B_{X^*}(\iota(x),\delta)\cap\iota[X]$ but, by density of $\iota[X]$, this is equivalent to $\{\iota(x)\}=B_{X^*}(\iota(x),\delta)$.
     \item Direct from (a) and (b).
  \end{enumerate}
\end{proof}

\begin{corollary}\label{PerfCode}
   $\la\omega,d\ra$ codes a perfect Polish space iff $\la\omega,d\ra$ is perfect. Even more, the set
   \[\Dwf^*(\omega):=\{d\in\Dwf(\omega):\la\omega,d\ra\textrm{ is perfect}\}\]
   is $\cosig^0_2$ in $(\R^\omega)^2$, so it is a Polish space.
\end{corollary}

Recall that every perfect countable metric space is homeomorphic to $\Q$, so all the codes for Perfect Polish spaces are pairwise homeomorphic.

Cantor-Bendixson Theorem (see, e.g., \cite[Thm. 6.4]{Kechris}) states that any Polish space has a unique partition on a perfect set and a countable open set. Even more, this perfect set is the largest closed perfect subset, usually known as the \emph{perfect kernel} of the space. More generally, using Cantor-Bendixson derivates, any second countable space has a perfect kernel (that is, a largest perfect closed subset) and its complement is countable (see \cite[Sect. 6.C]{Kechris}). However, the perfect kernel of a countable metric space does not represent the perfect kernel of its completion. For example, in $\R^2$, consider $D:=\{(\frac{1}{n+1},q_n):n<\omega\}$ where $\Q\cap(0,1)=\{q_n:n<\omega\}$ and let $X$ be the closure of $D$ in $\R^2$. Note that $X=D\cup(\{0\}\times[0,1])$ and that $D$ is open in $X$ and discrete. Thus, the perfect kernel of $D$ is the empty set, but the perfect kernel of $X$ is $X\menos D$.


\section{Coding continuous functions}\label{SecFunc}

The concept of Cauchy-continuous function is essential to code functions between Polish metric spaces. We review how a Cauchy-continuous function between metric spaces can be extended to a continuous function between their completions and also how can this process be reversed. The corresponding facts allows us to find an appropriate coding and its properties.

The following is a very useful tool to prove the results in this section.

\begin{lemma}\label{Cauchycontdense}
   Let $\iota:\la X_0,d_0\ra\to\la X_1,d_1\ra$ be a dense isometry between metric spaces and let $f:\la X_1,d_1\ra\to\la X_2,d_2\ra$ be a function between metric spaces. Then, $f$ is Cauchy-continuous iff $f$ is continuous and $f\circ\iota$ is Cauchy-continuous.
\end{lemma}
\begin{proof}
   Note that, if $\la X^*,d^*,\iota^*\ra$ is a completion of $\la X_1,d_1\ra$, then $\la X^*,d^*,\iota^*\circ\iota\ra$ is a completion of $\la X_0,d_0\ra$. For the implication from right to left, by Theorem \ref{ComplExt}, there exists a unique continuous function $\hat{f}:X^*\to\hat{X}_2$ such that $f\circ\iota=\hat{f}\circ\iota^*\circ\iota$ (here, wlog, we assume that $X_2$ is a dense subspace of its completion $\hat{X}_2$). As both $\hat{f}\circ\iota^*$ and $f$ are continuous functions on $X_1$ which coincide in $\iota[X_0]$ and this set is dense in $X_1$, then $f=\hat{f}\circ\iota^*$. Therefore, by Theorem \ref{ComplExt}(b), $f$ is Cauchy-continuous.
\end{proof}

Note that $f\circ\iota$ Cauchy-continuous does not imply $f$ continuous. For example, $f:[0,1]\to[0,1]$, defined as $f(x)=0$ if $x\in[0,1)$ and $f(1)=1$, is not continuous but its restriction to some dense subspace is Cauchy continuous, for example, on $(0,1)\cap\Q$.

The following result, on how to build functions between complete metric spaces from continuous functions between dense subspaces, can be seen as a particular case of Theorem \ref{ComplExt}.

\begin{theorem}\label{CompFunc}
  Let $\la X,d_X\ra$, $\la Y,d_Y\ra$ be metric spaces, and let $\la X^*,d^*_X,\iota_X\ra$ and $\la Y^*,d^*_Y,\iota_Y\ra$ be their respective completions. Let $f:\la X,d_x\ra\to\la Y,d_Y\ra$ be a continuous function.
  \begin{enumerate}[(a)]
     \item There is at most one continuous function $\hat{f}:\la X^*,d^*_X\ra\to\la Y^*,d^*_Y\ra$ such that $\iota_Y\circ f=\hat{f}\circ\iota_X$.
     \item $\hat{f}$ as in (a) exists iff $f$ is Cauchy-continuous.
     \item If $f$ is Cauchy-continuous, then
         \begin{enumerate}[(c-1)]
            \item $\hat{f}$ is uniformly continuous iff $f$ is.
            \item $\hat{f}$ is an isometry iff $f$ is.
            \item $\hat{f}$ is an isometrical isomorphism iff $f$ is a dense isometry.
         \end{enumerate}
     \item Assume that $f$ is Cauchy continuous. Let $\la X',d_{X'},\iota'_X\ra$ and $\la Y',d_{Y'},\iota'_Y\ra$ be completions of $\la X,d_X\ra$ and $\la Y,d_Y\ra$, respectively, and let $\iota_{X^*}:\la X^*,d_{X^*}\ra\to\la X',d_{X'}\ra$ and $\iota_{Y^*}:\la Y^*,d_{Y^*}\ra\to\la Y',d_{Y'}\ra$ be the isometrical isomorphisms such that $\iota'_X=\iota_{X^*}\circ\iota_X$ and $\iota'_Y=\iota_{Y^*}\circ\iota_Y$. If $\hat{f}:X^*\to Y^*$ and $\hat{f}':X'\to Y'$ are the continuous functions such that $\iota_Y\circ f=\hat{f}\circ\iota_X$ and $\iota'_Y\circ f=\hat{f}'\circ\iota'_X$, then $\hat{f}'=\iota_{Y^*}\circ\hat{f}\circ\iota_{X^*}^{-1}$.
  \end{enumerate}
\end{theorem}
\begin{proof}
   For (b) we use Theorem \ref{ComplExt}(b) and Lemma \ref{Cauchycontdense}. To see (d), note that $\hat{f}'\circ\iota_{X^*}\circ\iota_X=\hat{f}'\circ\iota'_X=\iota'_Y\circ f=\iota_{Y^*}\circ\iota_Y\circ f=\iota_{Y^*}\circ\hat{f}\circ\iota_X$, that is, $(\hat{f}'\circ\iota_{X^*})\circ\iota_X=(\iota_{Y^*}\circ\hat{f})\circ\iota_X$. Thus, as both $\hat{f}'\circ\iota_{X^*}$ and $\iota_{Y^*}\circ\hat{f}$ coincide on $\iota_X[X]$ and this set is dense in $X^*$, then we conclude that $\hat{f}'\circ\iota_{X^*}=\iota_{Y^*}\circ\hat{f}$.
\end{proof}

Theorem \ref{CompFunc}(d) indicates that any Cauchy-continuous function between metric spaces has a unique continuous extension (modulo isometrical isomorphisms) between their corresponding completions. The following result is a reciprocal of this.

\begin{lemma}\label{recCompFunc}
   Let $\la X,d_X\ra$ and $\la Y,d_Y\ra$ be metric spaces, let $\la X^*,d^*_X,\iota_X\ra$ and $\la Y^*,d^*_Y,\iota_Y\ra$ be their respective completions, and let $\hat{f}:X^*\to Y^*$ be a continuous function.
   \begin{enumerate}[(a)]
      \item There is a continuous function $f:\la X,d\ra\to \la Y,d\ra$ such that $\iota_Y\circ f=\hat{f}\circ\iota_X$ iff $\ran(\hat{f}\circ\iota_X)\subseteq\ran\iota_Y$. Moreover, such $f$ is unique, and it is Cauchy-continuous.
      \item There exists a cardinal $\eta\leq|X|+|Y|$, a metric $d'$ on $\eta$ and dense isometries $\iota:\la Y,d_Y\ra\to\la\eta,d'\ra$ and $\iota':\la\eta,d'\ra\to\la Y^*,d^*_Y\ra$ so that $\ran(\hat{f}\circ\iota_X)\subseteq\ran\iota'$.
   \end{enumerate}
\end{lemma}
\begin{proof}
   \begin{enumerate}[(a)]
      \item $\ran(\hat{f}\circ\iota_X)\subseteq\ran\iota_Y$ implies that $f:=\iota_Y^{-1}\circ\hat{f}\circ\iota_X:X\to Y$ is well defined and that $\iota_Y\circ f=\hat{f}\circ\iota_X$. Thus, by Theorem \ref{CompFunc}(b), $f$ is Cauchy-continuous. Uniqueness is straightforward, as well as the reciprocal.
      \item Put $Y'=\ran(f\circ\iota_X)\cup\ran\iota_Y$ and $\eta=|Y'|$. Choose $\iota':\eta\to Y'$ some bijection and let $d'$ be the metric on $\eta$ so that $\iota'$ becomes an isometry onto $\la Y',d^*_Y\frestr(Y'\times Y')\ra$. Note that $\iota=(\iota')^{-1}\circ\iota_Y$ works.
   \end{enumerate}
\end{proof}

The previous results guarantee that we can code continuous functions between Polish metric spaces by Cauchy-continuous functions between countable metric spaces.

\begin{definition}\label{Defcodecont}
   Define
   \[\Cwf(\omega):=\{(g,d,d'):d,d'\in\Dwf(\omega)\textrm{\ and }g:\la \omega,d\ra\to\la\omega,d'\ra\textrm{\ is Cauchy-continuous}\}.\]
   If $f:\la X,d_X\ra\to\la Y,d_Y\ra$ is a continuous function between infinite Polish metric spaces and $(g,d,d')\in\Cwf(\omega)$, say that \emph{$(g,d,d')$ codes $f$} if there are dense isometries $\iota:\la\omega,d\ra\to\la X,d_X\ra$ and $\iota':\la\omega,d'\ra\to\la Y,d_Y\ra$ such that $\iota'\circ g=f\circ\iota$.

   Define the relations $\leqcdi$ and $\eqcdi$ on $\Cwf(\omega)$ as follows (`cdi' stands for `commuting dense isometries'). $(g_0,d_0,d'_0)\leqcdi(g_1,d_1,d'_1)$ iff there are dense isometries $\iota:\la\omega,d_0\ra\to\la\omega,d_1\ra$ and $\iota':\la\omega,d'_0\ra\to\la\omega,d'_1\ra$ so that $\iota'\circ g_0=g_1\circ\iota$; $(g_0,d_0,d'_0)\eqcdi(g_1,d_1,d'_1)$ iff there is a $(g,d,d')\in\Cwf(\omega)$ so that $(g_e,d_e,d'_e)\leqcdi(g,d,d')$ for each $e=0,1$.
\end{definition}


According to the following result, the relation $\eqcdi$ determines whether two codes in $\Cwf(\omega)$ extend to the same continuous function.

\begin{lemma}\label{codesamecontfunct}
   For $e=0,1$ let $g_e:\la\omega,d_e\ra\to\la\omega,d'_e\ra$ be a Cauchy-continuous function between metric spaces. Then, the following statements are equivalent.
   \begin{enumerate}[(1)]
      \item Both $(g_0,d_0,d'_0)$ and $(g_1,d_1,d'_1)$ code the same continuous function, that is, there is a continuous function $f:\la X,d_X\ra\to\la Y,d_Y\ra$ between Polish metric spaces coded by both $(g_0,d_0,d'_0)$ and $(g_1,d_1,d'_1)$.
      \item $(g_0,d_0,d'_0)\eqcdi(g_1,d_1,d'_1)$.
   \end{enumerate}
\end{lemma}
\begin{proof}
   (2) implies (1) follows directly from Theorem \ref{CompFunc}. Assume (1), that is, for each $e=0,1$ there are dense isometries $\iota_e:\la\omega,d_e\ra\to\la X,d_X\ra$ and $\iota'_e:\la\omega,d'_e\ra\to\la Y,d_Y\ra$ so that $f\circ\iota_e=\iota'_e\circ g_e$. Put $Z:=\ran\iota_0\cup\ran\iota_1$, $Z':=\ran\iota'_0\cup\iota'_1$, choose bijections $\iota:\omega\to Z$, $\iota':\omega\to Z'$ and find $d,d'\in\Dwf(\omega)$ so that $d$ makes $\iota$ an isometry onto $\la Z,d_X\frestr(Z\times Z)\ra$ and $d'$ makes $\iota'$ an isometry onto $\la Z',d_Y\frestr(Z'\times Z')\ra$. For each $e=0,1$, put $\hat{\iota}_e:=\iota^{-1}\circ\iota_e:\la\omega,d_e\ra\to\la\omega,d\ra$ and $\hat{\iota}'_e:=\iota^{-1}\circ\iota_e:\la\omega,d'_e\ra\to\la\omega,d'\ra$ which are dense isometries. Also, $\iota\circ\hat{\iota}_e=\iota_e$ and $\iota'\circ\hat{\iota}'_e=\iota'_e$. On the other hand, $\ran(f\circ\iota)=\ran(f\circ\iota_0)\cup\ran(f\circ\iota_1)=\ran(\iota'_0\circ g_0)\cup\ran(\iota'_1\circ g_1)\subseteq\ran\iota'$ so, by Lemma \ref{recCompFunc}(b), there is a Cauchy-continuous $g:\la\omega,d\ra\to\la\omega,d'\ra$ so that $\iota'\circ g=f\circ\iota$. Then, we can infer that $\iota'\circ g\circ\hat{\iota}_e=\iota'\circ\hat{\iota}'_e\circ g_e$ for each $e=0,1$, so $g\circ\hat{\iota}_e=\hat{\iota}'_e\circ g_e$. Therefore, $(g_e,d_e,d'_e)\leqcdi(g,d,d')$.
\end{proof}

We also provide the complexity of $\eqcdi$ and of other related statements.

\begin{theorem}\label{codesamecontfunctAnalytic}
  \begin{enumerate}[(a)]
     \item The statement ``$z$ is a Cauchy sequence in the metric space $\la\omega,d\ra$" is $\cosig^0_3$ in $\omega^\omega\times\R^{\omega\times\omega}$.
     \item The statement ``$g:\la\omega,d\ra\to\la\omega,d'\ra$ is continuous between metric spaces" is $\cosig^0_3$ in $\omega^\omega\times(\R^{\omega\times\omega})^2$.
     \item $\Cwf(\omega)$ is $\cosig^1_1$ in $\omega^\omega\times(\R^{\omega\times\omega})^2$.
     \item The relation $\leqcdi$ in $\Cwf(\omega)$ is a conjunction of a $\sig^1_1$ with a $\cosig^1_1$ relation in $(\omega^\omega)^2\times(\R^{\omega\times\omega})^4$.
     \item The relation $\eqcdi$ in $\Cwf(\omega)$ is a conjunction of a $\sig^1_1$ with a $\cosig^1_1$ statement in $(\omega^\omega)^2\times(\R^{\omega\times\omega})^4$.
  \end{enumerate}
\end{theorem}
\begin{proof}
   We only focus on (e). Note that, for $(g_0,d_0,d'_0),(g_1,d_1,d'_1)\in\Cwf(\omega)$, $(g_0,d_0,d'_0)\eqcdi(g_1,d_1,d'_1)$ iff `there are $d,d'\in\Cwf(\omega)$, dense isometries $\iota_e:\la\omega,d_e\ra\to\la\omega,d\ra$ and $\iota'_e:\la\omega,d'_e\ra\to\la\omega,d'\ra$ for each $e=0,1$, and there is a continuous function $g:\la\omega,d\ra\to\la\omega,d'\ra$ such that $\iota'_e\circ g_e=g\circ\iota_e$ for each $e=0,1$' because such $g$ must be Cauchy-continuous by Lemma \ref{Cauchycontdense}. This latter statement is analytic by (a)-(c) and Theorem \ref{codesamePolishAnalytic}. Therefore, the relation $\eqcdi$ is a conjunction of the previous analytic statement with the co-analytic statement `$(g_0,d_0,d'_0)\in\Cwf(\omega)$ and $(g_1,d_1,d'_1)\in\Cwf(\omega)$'.
\end{proof}

Finally, thanks to the results of this section, we can characterize when two countable metric spaces code the same Polish space (that is, homeomorphic Polish spaces) and we also find the complexity of this equivalence relation.

\begin{theorem}\label{codesamePolish}
   Let $d_0,d_1\in\Dwf(\omega)$. Then, $\la\omega,d_0\ra$ and $\la\omega,d_1\ra$ have homeomorphic completions iff there are $d'_0,d'_1\in\Dwf(\omega)$ such that $d_e\leqdi d'_e$ for each $e=0,1$ and there is a Cauchy-continuous bijection $g:\la\omega,d'_0\ra\to\la\omega,d'_1\ra$ with Cauchy-continuous inverse.
\end{theorem}
\begin{proof}
   Assume that $d'_0,d'_1\in\Dwf(\omega)$ satisfy $d_e\leqdi d'_e$ for each $e=0,1$ and that there is a Cauchy-continuous bijection $g:\la\omega,d'_0\ra\to\la\omega,d'_1\ra$ with Cauchy-continuous inverse. Choose a completion $\la X^*_e,d^*_e,\iota^*_e\ra$ of $\la\omega,d'_e\ra$ (which also yields a completion of $\la\omega,d_e\ra$) for each $e=0,1$. By Theorem \ref{CompFunc} applied to $g$ and to $g^{-1}$, there are continuous functions $f^*_0:X_0^*\to X_1^*$ and $f^*_1:X_1^*\to X_0^*$ such that $\iota^*_1\circ g=f^*_0\circ\iota^*_0$ and $\iota^*_0\circ g^{-1}=f^*_1\circ\iota^*_1$. Thus, \[\iota^*_0\circ\id_\omega=\iota^*_0\circ g^{-1}\circ g=f^*_1\circ\iota^*_1\circ g=(f^*_1\circ f^*_0)\circ\iota^*_0,\]
   so, by Theorem \ref{CompFunc} (uniqueness), $f^*_1\circ f^*_0=\id_{X^*_0}$. Conversely, $f^*_0\circ f^*_1=\id_{X^*_1}$, so $X^*_0$ and $X^*_1$ are homeomorphic.

   To see the converse, let $\la X'_e,d'_e,\iota'_e\ra$ be a completion of $\la\omega,d_e\ra$ for each $e=0,1$ and assume that there is an homeomorphism $f:X'_0\to X'_1$. Put $D_0=\ran\iota'_0\cup\ran(f^{-1}\circ\iota'_1)$ and $D_1=f[D_0]=\ran(f\circ\iota'_0)\cup\ran\iota'_1$. For each $e=0,1$, as in the proof of Lemma \ref{recCompFunc}(b), find a $d'_e\in\Dwf(\omega)$ such that there is an isometrical isomorphism $\iota^*_e:\la\omega,d'_e\ra\to D_e$. Define $\iota_e=(\iota^*_e)^-1\circ\iota'_e$, which is clearly an dense isometry from $\la\omega,d_e\ra$ to $\la\omega,d'_e\ra$, so $d_e\leqdi d'_e$. By Lemma \ref{recCompFunc}(a) applied to $f$ and $f^{-1}$, there are Cauchy-continuous functions $g:\la\omega,d'_0\ra\to\la\omega,d'_1\ra$ and $g':\la\omega,d'_1\ra\to\la\omega,d'_0\ra$ such that $f\circ\iota^*_0=\iota^*_1\circ g$ and $f^{-1}\circ\iota^*_1=\iota^*_0\circ g'$. As
   \[\id_{X'_0}\circ\iota^*_0=f^{-1}\circ f\circ \iota^*_0=f^{-1}\circ\iota^*_1\circ g=\iota^*_0\circ(g'\circ g),\]
   by uniqueness in Lemma \ref{recCompFunc}(a), $g'\circ g=\id_\omega$. Likewise, we obtain $g\circ g'=\id_\omega$, so $g$ is bijective and $g^{-1}=g'$ is Cauchy-continuous.
\end{proof}

We denote the relation in the previous theorem by $d_0\eqPc d_1$ (`P' stands for `Polish'), which means that $\la\omega,d_0\ra$ and $\la\omega,d_1\ra$ code homeomorphic Polish spaces. As in Theorem \ref{codesamecontfunctAnalytic}, it is easy to see that being a Cauchy-continuous bijection with a Cauchy-continuous inverse is a co-analytic statement. Therefore,

\begin{corollary}\label{codesamePolishSigma1-2}
  The relation $\eqPc$ is $\sig^1_2$ in $(\R^{\omega\times\omega})^2$.
\end{corollary}

\subsection*{Acknowledgements} This paper was produced for the conference proceedings of the RIMS Workshop on Mathematical Logic and Its Applications which was held in the last week of September of 2016. The author is very thankful with professor Makoto Kikuchi for organizing this great workshop and for letting him participate as a speaker. The author also wants to thank Miguel Cardona for pointing out the last example in Section \ref{SecPolish}.

\bibliography{left}
\bibliographystyle{alpha}


\end{document}